\documentclass[preprint,3p,sort]{elsarticle}

\usepackage{amsfonts}
\usepackage{amsthm}
\usepackage{graphicx}
\usepackage{epstopdf}
\usepackage{algorithmic}
\usepackage{amsmath}
\usepackage{amssymb}
\usepackage{hyperref}
\usepackage[T1]{fontenc}
\newtheorem{theorem}{Theorem}[section]
\newtheorem{proposition}[theorem]{Proposition}

\newtheorem{definition}{Definition}[section]
\theoremstyle{remark}
\newtheorem{remark}[theorem]{Remark}
\numberwithin{equation}{section}
\begin{document}


\begin{frontmatter}

%

\title{Nonlinear and Linearised Primal and Dual Initial Boundary Value Problems: When are they Bounded? How are they Connected?}

\author[sweden,southafrica]{Jan Nordstr\"{o}m}
\cortext[secondcorrespondingauthor]{Corresponding author}
\ead{jan.nordstrom@liu.se}
\address[sweden]{Department of Mathematics, Applied Mathematics, Link\"{o}ping University, SE-581 83 Link\"{o}ping, Sweden}
\address[southafrica]{Department of Mathematics and Applied Mathematics, University of Johannesburg, P.O. Box 524, Auckland Park 2006, Johannesburg, South Africa}

\begin{abstract}
Linearisation is often used as a first step in the analysis of nonlinear initial boundary value problems. 
The linearisation procedure frequently results in a confusing contradiction where the nonlinear problem conserves energy and has an energy bound but the linearised version does not (or vice versa). In this paper we attempt to resolve that contradiction and relate nonlinear energy conserving and bounded initial boundary value problems to their linearised versions and the related dual problems. We start by showing that a specific skew-symmetric form of the primal nonlinear problem leads to energy conservation and a bound. Next, we show that this specific form together with a non-standard linearisation procedure preserves these properties for the new slightly modified linearised problem. We proceed to show that the corresponding linear and nonlinear dual (or self-adjoint) problems also have bounds and conserve energy due to this specific formulation. Next, the implication of the new formulation on the choice of boundary conditions is discussed. A straightforward nonlinear and linear analysis may lead to a different number and type of boundary conditions required for an energy bound. We show that the new formulation shed some light on this contradiction. We conclude by illustrating that the new continuous formulation automatically lead to energy stable and energy conserving numerical approximations for both linear and nonlinear primal and dual problems if the approximations are formulated on summation-by-parts form.
\end{abstract}

\begin{keyword}
initial boundary value problems \sep energy estimates \sep skew-symmetric form \sep linearisation procedure \sep nonlinear problems \sep linear problems \sep dual problems \sep  self-adjoint problems \sep nonlinear self-adjointness \sep variable coefficients \sep boundary conditions \sep summation-by-parts \sep energy stability
\end{keyword}


\end{frontmatter}


\section{Introduction}

The classical way of analysing linear initial boundary value problems (IVBPs) uses the energy method \cite{kreiss1970,kreiss1989initial,Gustafsson1978,gustafsson1995time,oliger1978} as the main tool. The energy method applied to linear IBVPs lead to well posed boundary conditions and energy estimates \cite{nordstrom2020,nordstrom_roadmap,nordstrom2005,ghader2014}. This procedure can also be used directly for certain types of nonlinear problems as recently shown in \cite{nordstrom2019, nordstrom2020spatial, Lauren2021}. Another common procedure to obtain estimates for nonlinear problems is to use the entropy stability theory \cite{tadmor1984,Tadmor1987,Tadmor2003}  which originated in \cite{godunov1961interesting,volpert1967,kruzkov1970,dafermos1973entropy,lax1973,harten1983} and is applied for example in \cite{dubois1988,hindenlang2019,parsani2015entropy,svard2012,svard2021entropy}.

In this paper we will not focus on nonlinear problems per se, but rather on the relation between a nonlinear problem, and its linearised counterpart, the variable coefficient problem. This relation is interesting for many reasons, perhaps the most important one (except for being an important analysis tool) is that it constitutes the starting point for the derivation of the dual (or adjoint) problem often used in optimisation \cite{Jameson1988233,Jameson1998213,Giles2000393,Giles2002145,Nielsen20021155,Fidkowski2011673}. Although we study energy estimates for nonlinear and linearised problems which eventually lead to energy stable schemes, our analysis may also shed some light on the linear stability problems recently observed for nonlinear entropy stable schemes \cite{gassner2020stability,ranocha2021}. A nonlinear problem connects to its linearised problem through the linearisation procedure which frequently results in confusing contradictions. Given that the nonlinear problem conserves energy and has an energy bound, the resulting linearised variable coefficient version often does not (or vice versa). 
In addition to being confusing, this leads to practical problems, namely: ill-posedness of the corresponding dual problems \cite{berg2012,berg2013,berg2014,nordstrom_dual_2017,Nordstrom-Ghasemi2020}.


The observed ill-posedness of the dual problem typically manifests itself as a numerical stability problem (sometimes referred to as "butterfly effects" \cite{Thalabard2020,Lohner2014742,Wang20131,wilcox2013}) and various stabilisation techniques (often referring to the famous Lorenz 63 system \cite{Lorenz63}) have been developed, see for example \cite{Wang20131,Wang2014210,Wang2014156,Lea2000523,Eyink20041867,Thuburn200573,Blonigan2014660}. Although we will not deal specifically with this issue, we note that the acclaimed Lorenz 63 system deals with ordinary differential equations, where the stabilising effect of correctly imposed boundary conditions in IBVPs is missing. The analysis in this paper indicate that some of these stabilisation techniques may not be required, and that both the linear and nonlinear dual problems have energy bounds if properly posed.

Nonlinear dual problems have received considerable interest recently. The general concept of nonlinear self-adjointness \cite{Ibragimov2006742,NEWIbragimov2011}, which includes strict self-adjointness \cite{Ibragimov2006742,Ibragimov2007311}, quasi self-adjointness \cite{S009630031200667420121001} and weak self-adjointness \cite{edselc.2-52.0-7995823472520110701}, was originally introduced to construct conservation laws associated with certain symmetry properties of differential equations  \cite{OLDIbragimov2011,Zhang2013,edselc.2-52.0-7995823472520110701}. Nonlinear self-adjointness has also been used to derive exact transformations between linear and nonlinear formulations \cite{Tracina2015}. We will also consider nonlinear self-adjoint problems, but our focus is different. We will consider the relation between nonlinear problems and related linear ones where the linear problems are obtained using the standard linearisation procedure which neglects quadratic terms  \cite{Strang196437,kreiss1989initial,gustafsson1995time}. 

Our ambition is to derive energy conserving and bounded formulations for both the nonlinear and linearised problem, later to be transformed to energy stable numerical approximations. 
We focus on the formulation of the continuous equations, but also discuss boundary conditions. 
A nonlinear and linear analysis may lead to a different number and type of boundary conditions required for an energy bound. This was shown in \cite{nordstrom2021linear} for the nonlinear and linearised shallow water equations. There are also cases where this discrepancy does not occur, see \cite{nordstrom2019, nordstrom2020spatial, Lauren2021,Nordstrom2007874} for examples. We show that the new skew-symmetric formulation presented herein  shed some light on this issue.

Using lifting approaches \cite{Arnold20011749,nordstrom_roadmap} and proper continuous boundary conditions, it is straightforward to apply the results (formulations that lead to energy conservation and bounds) from the continuous analysis and develop stable numerical schemes. This procedure enables research groups using different numerical techniques  such as finite difference \cite{nordstrom2009stable,svard2007stable,svard2008stable}, finite volume \cite{nordstrom2012weak,nordstrom2003finite}, spectral elements \cite{carpenter2014entropy,carpenter1996spectral}, flux reconstruction \cite{castonguay2013energy,huynh2007flux,huynh2007flux}, discontinuous Galerkin \cite{gassner2013skew,hesthaven1996stable,kopriva2021} and continuous Galerkin schemes \cite{abgrall2020analysis,abgrall2021analysis} to make use of the results. The only requirement is that one can formulate the numerical procedure on summation-by-parts (SBP) form with weak boundary conditions on simultaneous approximation term (SAT) form \cite{svard2014review,fernandez2014review} or equivalently through numerical flux functions \cite{kopriva2021}. We conclude the paper by exemplifying how both linear and nonlinear stability follows  almost  automatically using the SBP-SAT technique with proper boundary conditions. 

The paper is organised as follows: We start by illustrating the confusing contradiction mentioned above for the scalar Burgers' equation in Section~\ref{sec:burgers}. This sets the stage for the discussion in Section~\ref{sec:energyStab} regarding which form a general, provably bounded nonlinear or linear hyperbolic IBVP should have. By exploiting a particular skew-symmetric form, we show how to relate the linearised problem to the nonlinear one in Section~\ref{sec:newlinear}. Section~\ref{sec:dual} presents the nonlinear and linear dual problems and show that they also have energy bounds and preserve energy. In Section~\ref{sec:examples} we return to Burgers' equation and  show how the initial contradiction discussed in Section~\ref{sec:burgers} is resolved. In addition we discuss more complex examples involving systems of IBVPs. In Section~\ref{boundary_conditions}, we discuss the relation between a linear and nonlinear boundary treatment. Section~\ref{numerics} illustrate the close relation between the skew-symmetric continuous formulation and stability and energy conservation of the numerical scheme. A summary and conclusions are provided in Section~\ref{sec:conclusion}.

\section{The standard linearisation procedure}\label{sec:burgers}

We start by considering the one-dimensional (1D) hyperbolic (inviscid) Burgers' equation. The primitive, conservative and skew-symmetric forms  are 
\begin{equation}\label{eq:burgers}
u_t + u u_x=0,  \quad u_t +  \left( \frac{u^2}{2}  \right)_x=0,  \quad u_t + \frac{1}{3}(u u_x +(u^2)_x) =0
\end{equation}
respectively. Here $u(x,t)$ is the solution where $x \in  \Omega=\lbrack a,b \rbrack$ and $t \geq 0 $ are the spatial and time coordinates. We also require that no discontinuities  are present (which could be accomplished by adding suitable dissipative terms).  With a smooth solution, all formulations in (\ref{eq:burgers}) are mathematically equivalent.

\subsection{Energy bound in the nonlinear Burgers' equation}\label{sec:burgersstandlin}

By applying the energy method (multiplying the equation by the solution and integrating over the domain), the formulations in (\ref{eq:burgers}) all lead to
\begin{equation}\label{eq:burgersenergy}
\frac{d}{dt}\|u\|^2_2 + \frac{2}{3} u^3|^b_a =0,
\end{equation}
i.e. the energy $\|u\|^2_2= \int^b_a u^2 dx$ grows or decays only through boundary effects, which can be controlled by suitable boundary conditions. Note in particular that no volume terms contribute to the growth or decay.

\subsection{Growth in the linearised Burgers' equation}\label{sec:burgersstandlin_sec}

Next we proceed in the standard way and make the ansatz $u=\bar u (x,t) +  u^\prime (x,t)$ where $\bar u$  is $\cal{O}$(1) while $u^\prime$ is a small pertubation ($ |u^\prime | \ll |\bar u|$). By inserting the ansatz into (\ref{eq:burgers})  we find
\begin{equation}\label{eq:burgersnonlin}
\bar u_t + \bar u \bar u_x+u^\prime_t+\bar u u^\prime_x +  \bar u_x u^\prime+u^\prime u^\prime_x =0.
\end{equation}
In the standard linearisation procedure one cancels the last quadratic term due to it being negligible, and assumes \cite{Strang196437,kreiss1989initial,gustafsson1995time} that
\begin{equation}\label{eq:governeananddist}
    \bar u_t + \bar u \bar u_x=0 \quad \text{and} \quad u^\prime_t+\bar u u^\prime_x +  \bar u_x u^\prime=0
\end{equation}
holds.  

This procedure leads to two confusing results: the first equation in (\ref{eq:governeananddist}), which states that $\bar u$ satisfies the original governing equation (\ref{eq:burgers}) can, strictly speaking, only hold if $u^\prime = 0$ (disregarding the trivial case where  $\bar u$ is constant and assuming uniqueness). Applying the energy method to the second equation yields
\begin{equation}\label{eq:burgersenergydist}
\frac{d}{dt}\|u^\prime \|^2_2 +  \bar u (u^\prime)^2|^b_a =-\int^b_a \bar u_x (u^\prime)^2 dx \leq  \max_{a  \leq  x  \leq  b}\lvert \bar u_x \rvert  \|u^\prime \|^2_2.
\end{equation}
The volume term on the right-hand side may cause  exponential growth (or decay) of energy, even if suitable boundary conditions are supplied. The variable coefficient $\bar u_x$ produces the volume term. No such effect exist in the nonlinear problem. This initial example illustrates that even though the nonlinear problem has an energy bound and conserves energy as in (\ref{eq:burgersenergy}), the linearised problem may not, as seen in (\ref{eq:burgersenergydist}).
\begin{remark}
It is straightforward to show that the second equation in (\ref{eq:governeananddist}) also leads to growth in the dual
problem. The growth (or decay) in these linear problems remains even if the zero order term is dropped (it only leads to a sign change of the volume term in (\ref{eq:burgersenergydist})). Only the trivial case $\bar u_x=0$ has no volume growth.
\end{remark}

\section{Energy bounded and energy conserving linear and nonlinear primal problems}\label{sec:energyStab}

The results above for the Burgers' equation show that a general formulation that allows for an energy estimate for both the nonlinear and linearised equations would be of interest. Throughout this paper we restrict the analysis to the hyperbolic (inviscid) part of IBVPs, where the nonlinearity normally resides. The parabolic (viscous) part could be added on and properly posed provide dissipation or damping effects (although they can complicate the boundary treatment \cite{kreiss1989initial,Gustafsson1978,nordstrom2020,nordstrom2005,nordstrom2019,nordstrom2020spatial,Lauren2021}).

Consider the following general hyperbolic IBVP
\begin{equation}\label{eq:nonlin}
P U_t + (A_i(V) U)_{x_i}+B_i(V)U_{x_i}+C(V)U=0,  \quad t \geq 0,  \quad  \vec x=(x_1,x_2,..,x_k) \in \Omega
\end{equation}
augmented with homogeneous boundary conditions $L_p U=0$ at the boundary $\delta \Omega$. In (\ref{eq:nonlin}),  the Einstein summation convention is used and $P$ is a symmetric positive definite (or semi-definite) time-independent matrix that defines an energy norm (or semi-norm) $\|U\|^2_P= \int_{\Omega} U^T P U d\Omega$. We assume that $U$ and $V$ are smooth. The $n \times n$ matrices $A_i,B_i,C$ are smooth functions (each matrix element is smooth) of the $n$ component long vector $V$, but otherwise arbitrary. Note that (\ref{eq:nonlin}) encapsulates both linear ($V \neq U$) and nonlinear  ($V=U$) problems. The following two concepts are essential for a proper treatment of (\ref{eq:nonlin}).
\begin{definition}
The problem (\ref{eq:nonlin}) is energy conserving if $\|U\|^2_P= \int_{\Omega} U^T P U d\Omega$ only changes due to boundary effects. It is energy bounded if $\|U\|^2_P < \infty$ as $t  \rightarrow  \infty$.
\end{definition}
\begin{proposition}\label{lemma:Matrixrelation}
The IBVP  (\ref{eq:nonlin}) for linear ($V \neq U$) and nonlinear  ($V=U$)  is energy conserving if
\begin{equation}\label{eq:boundcond}
B_i= A_i^T, \quad i=1,2,..,k \quad \text{and } \quad C+C^T = 0
\end{equation}
holds. It is energy bounded if it is energy conserving and the boundary conditions $L_p U=0$ are such that 
\begin{equation}\label{1Dprimalstab}
\oint\limits_{\partial\Omega}U^T  (n_i A_i)   \\\ U \\\ ds = \oint\limits_{\partial\Omega} \frac{1}{2} U^T ((n_i A_i)  +(n_i A_i )^T) U \\\ ds \geq 0.
\end{equation}
\end{proposition}
\begin{proof}
The energy method applied to (\ref{eq:nonlin}) yields
\begin{equation}\label{eq:boundaryPart1}
\frac{1}{2} \frac{d}{dt}\|U\|^2_P + \oint\limits_{\partial\Omega}U^T  (n_i A_i)  \\\ U \\\ ds= \int\limits_{\Omega}(U_{x_i}^T  A_i U - U^T B_i U_{x_i}) \\\ d \Omega -\int\limits_{\Omega} U^T  C U \\\ d \Omega,
\end{equation}
where $(n_1,..,n_k)^T$ is the outward pointing unit normal. The right-hand side of (\ref{eq:boundaryPart1}) is cancelled by (\ref{eq:boundcond}) leading to energy conservation. If in addition (\ref{1Dprimalstab}) holds, an energy bound is obtained.
\end{proof}
\begin{remark}
The relation (\ref{eq:boundcond})  leads a to skew-symmetric formulation, where symmetric matrices \cite{Abarbanel19811, nordstrom_roadmap,oliger1978,nordstrom2005} are allowed, but not required for the energy method to work. Proposition \ref{lemma:Matrixrelation}  shows that whatever form the original IBVP has, energy boundedness and energy conservation can be proved if it can be rewritten in the form given by (\ref{eq:nonlin})-(\ref{eq:boundcond}). 
\end{remark}


For later reference we introduce the  linear ($V \neq U$) and nonlinear ($V=U$) primal problem to be analysed
\begin{equation}\label{1Dprimal}
\begin{aligned}
P U_t + (A_i(V) U)_{x_i}+A_i^T(V)U_{x_i}+C(V)U=\	&F,& \vec x&\in \Omega,& t&\geq 0\\
L_p U=\	&g,& \vec x&\in \partial\Omega,& t&\geq 0\\
U =\	&f,&\vec x&\in\Omega,& t&= 0,
\end{aligned}
\end{equation}
where $F$ is a forcing function, $g$ is boundary data and $f$ initial data.
\begin{remark}
The skew-symmetric form in (\ref{1Dprimal}) leads to a complete derivative in the energy method since
\[
U^T (A_i U)_{x_i}+U^T A_i^T U_{x_i}=U^T (A_i U)_{x_i}+(A_i U)^TU_{x_i}=U^T (A_i U)_{x_i}+U_{x_i}^T(A_i U)=(U^T A_i  U)_{x_i}.
\]
\end{remark}
\begin{remark}
Inhomogeneous boundary conditions $L_p U=g$ will be discussed in Section~\ref{boundary_conditions}. 
\end{remark}

\section{A new non-standard linearisation procedure}\label{sec:newlinear}
%
In the first step, we divide the solution into a non-constant mean value $\bar U$ and small perturbation $U^\prime$  as $V=U=\bar U +  U^\prime$ and insert this into the homogeneous version of (\ref{1Dprimal}). 
Next we introduce perturbed matrices $A_i(\bar U +  U^\prime)=\bar A_i(\bar U) +  A^\prime_i(\bar U,U^\prime)$, where $A^\prime_i$ vanishes as $U^\prime$ vanishes and similary for $C(\bar U +  U^\prime)=\bar C(\bar U) +  C^\prime(\bar U,U^\prime)$. 
 This partition can always be done, since  the matrix entries are smooth. The result is
\begin{equation}\label{eq:nonlin1d}
P \bar U_t + P U^\prime_t +(\bar A_i \bar U + \bar A_i  U^\prime + A^\prime_i \bar U)_{x_i}+\bar A^T_i \bar U_{x_i}+\bar A^T_i U^\prime_{x_i} + (A^\prime_i)^T \bar U_{x_i}+\bar C \bar U + \bar C U^\prime + C^\prime \bar U+H=0,
\end{equation}
where $H=\lbrack (A^\prime_i U^\prime)_{x_i} + (A^\prime_i)^T U^\prime_{x_i} + C^\prime U^\prime \rbrack =\mathcal{O}(|U^\prime|^2)$ collects the nonlinear terms.  
The next step is interpreted in slightly different ways by different authors  \cite{Strang196437,kreiss1989initial,gustafsson1995time} as pointed out above. In the standard version one neglects $H$ and requires that
\begin{equation}\label{eq:linearisedtilde}
P \bar U_t + (\bar A_i  \bar U)_{x_i}+\bar A ^T_i \bar U_{x_i}+\bar C \bar U=0
\end{equation}
holds. However, this interpretation is doubtful since $\bar U$ can only solve (\ref{eq:linearisedtilde}) if $U^\prime$ vanishes as seen in (\ref{eq:nonlin1d}). 

By carefully considering Proposition \ref{lemma:Matrixrelation} we instead group the terms in the following way:
\begin{equation}\label{eq:linearised1d}
\underbrace{P \bar U_t + ((\bar A_i + A^\prime_i) \bar U)_{x_i}+(\bar A_i + A^\prime_i)^T \bar U_{x_i}+(\bar C +  C^\prime) \bar U}_{mean} + \underbrace{P U^\prime_t +(\bar A_i  U^\prime)_{x_i}+\bar A^T_i U^\prime_{x_i}+\bar C U^\prime}_{perturbation}+H=0.
\end{equation}
A more plausible interpretation is thus obtained by neglecting $H$ and demanding that
\begin{equation}\label{eq:linearisedtildenonhom_v1}
P \bar U_t + ((\bar A_i + A^\prime_i) \bar U)_{x_i}+(\bar A_i + A^\prime_i)^T \bar U_{x_i}+(\bar C +  C^\prime) \bar U=0,
\end{equation}
or equivalently 
\begin{equation}\label{eq:linearisedtildenonhom}
P \bar U_t + (A_i \bar U)_{x_i}+A_i ^T \bar U_{x_i}+C \bar U=0,
\end{equation}
which leaves the linearised primal equation to be 
\begin{equation}\label{eq:linearisedfinal}
P U^\prime_t +(\bar A_i U^\prime)_{x_i}+\bar A^T_i U^\prime_{x_i}+\bar C U^\prime=0.
\end{equation}
Equations (\ref{eq:linearisedtildenonhom}) and (\ref{eq:linearisedfinal}) are now both in the skew-symmetric form required in Proposition \ref{lemma:Matrixrelation} for energy boundedness. This formulation removes the common confusing growth issue of the linearised problem, when the nonlinear original problem (\ref{1Dprimal}) has a bound. Note also that (\ref{eq:linearisedtildenonhom}) and (\ref{eq:linearisedfinal}) are coupled.

\section{Energy bounded and energy conserving linear and nonlinear dual problems}\label{sec:dual}
Consider the functional
\begin{equation*}
J(U,G)=\int_{\Omega}U^TG\	d\Omega,
\end{equation*}
where $G$ is a vector weight function. To derive the dual problem to (\ref{1Dprimal}), we search for the dual solution $\Phi$ such that $J(U,G)=J(\Phi,F)$, where $F$ is the forcing function in (\ref{1Dprimal}).  Integration by parts yields
\begin{equation*}
\begin{aligned}
\int_{0}^{T}J(U,G)dt&=\int_{0}^{T}J(U,G) dt-\int_{0}^{T}(\Phi,P U_t + (A_i(V) U)_{x_i}+A^T_i(V) U_{x_i}+C(V)U-F)\	dt\\&=
\int_{0}^{T}J(\Phi,F)\	dt-(\Phi,PU)\bigg|_{0}^{T}-\int_{0}^{T}\oint_{\partial\Omega}\Phi^T (n_i A_i) U\	dsdt\\&+\int_{0}^{T}(P \Phi_t+(A_i\Phi)_{x_i}+A^T_i \Phi_{x_i}+C \Phi+G,U)\	dt.
\end{aligned}
\end{equation*}
With homogeneous initial condition for the primal problem, it follows that the dual end condition is $\Phi(\vec x,T)=0$. Furthermore, the dual homogeneous boundary conditions $L_d\Phi=0$  must be such that  $\Phi^T (n_i A_i)  U=0, \vec x \in \partial\Omega$ when the primal homogeneous boundary conditions are applied.
Finally, the dual equation becomes
\begin{equation}\label{define_self}
-P \Phi_t-(A_i(V)\Phi)_{x_i}-A(V)^T_i \Phi_{x_i}-C(V) \Phi=G.
\end{equation}\par
The derivation of the dual problem (\ref{define_self}) holds both in the linear ($V \neq U$) and nonlinear  ($V = U$) cases. By setting $G=0$ and letting $V = \Phi$ we find 
that the dual problem is identical to the primal problem (\ref{1Dprimal}) with $F=0$, and hence it is nonlinearly self-adjoint, or more precisely strictly self-adjoint \cite{Ibragimov2006742,Ibragimov2007311,NEWIbragimov2011}. 
\begin{remark}
The concept of nonlinear self-adjointness was originally introduced to aid construction of conservation laws associated with certain symmetry properties of differential equations. 
One may speculate (no proof exists) that many conservation laws could possibly be transformed to the skew-symmetric form in Proposition \ref{lemma:Matrixrelation} and proven energy bounded. This provides interesting future research as many practical problems are formulated as conservation laws, but lack a proof of boundedness. 
\end{remark}


By introducing the time transformation $\tau=T-t$, both the linear ($V \neq \Phi $) and nonlinear ($V=\Phi $) dual problems become
\begin{equation}\label{1Ddual}
\begin{aligned}
P \Phi_\tau-(A_i(V)\Phi)_{x_i}-A(V)^T_i \Phi_{x_i}-C(V) \Phi=\	&G,& \vec  x&\in \Omega,& \tau&\geq 0\\
L_d\Phi=\	&q,& \vec  x&\in \partial\Omega,& \tau&\geq 0\\
\Phi=\	&r,&\vec  x&\in\Omega,& \tau&= 0
\end{aligned}
\end{equation}
where $G$ is the forcing function, $q$ is boundary data and $r$ initial data.

We state the corresponding results for the dual problem that hold for the primal one. 
\begin{proposition}\label{lemma:Matrixrelation_dual}
The linear ($V \neq \Phi$) and nonlinear ($V=\Phi$) version of the IBVP (\ref{1Ddual}) is energy conserving if $C+C^T = 0$ holds. It is energy bounded if it is energy conserving and the boundary conditions $L_d\Phi=0$ are such that
\begin{equation}\label{1Ddualstab}
\oint_{\partial\Omega}\Phi^T (n_i A_i)  \Phi ds = \oint_{\partial\Omega}  \frac{1}{2} \Phi^T ((n_i A_i) +(n_i A_i) ^T) \Phi ds \leq 0.
\end{equation}
\end{proposition}
\begin{proof}
The proof is identical to the one for Proposition (\ref{lemma:Matrixrelation}).
\end{proof}



\section{Examples}\label{sec:examples}
We will discuss four examples of increasing complexity. We start by again considering Burgers' equation. 
\subsection{The 1D scalar Burgers' equation}\label{Burgersex}

Any of the formulations in (\ref{eq:burgers}) can be used to arrive at the skew-symmetric (or self-adjoint) forms below, but perhaps starting with the confusing formulation in 
(\ref{eq:burgersnonlin}) is the most instructive way. We have 
\begin{equation}\label{eq:burgersnonlinmodif}
\bar u_t + \bar u \bar u_x+u^\prime_t+\bar u u^\prime_x +  \bar u_x u^\prime+h=0,
\end{equation}
where $h=u^\prime u^\prime_x $. By collecting terms based on the formulation (\ref{eq:linearised1d}) we find
\begin{equation}\label{skew-sym-burgers}
\underbrace{\bar u_t + \frac{1}{3}((\bar u + u^\prime) \bar u)_x+\frac{1}{3}(\bar u + u^\prime) \bar u_x }_{mean} + \underbrace{u^\prime_t +\frac{1}{3}(\bar u  u^\prime)_x+\frac{1}{3}\bar u u^\prime_x}_{perturbation}+h=0.
\end{equation}
After neglecting the nonlinear term $h$ we set both groups to zero and obtain the equations for the mean and the perturbation. This leads to skew-symmetric self-adjoint formulations for both the mean and the perturbation and the preceding theory in Sections \ref{sec:energyStab}, \ref{sec:newlinear} and \ref{sec:dual} applies.

\subsection{The 2D incompressible Euler equations}\label{Eulerex}

The incompressible 2D Euler equations in primitive form with velocity field $(u,v)$ in the $(x,y)$ direction and pressure $p$ (divided by the constant density) are
\begin{align}
u_t + u u_x + v u_y + p_x &= 0,\nonumber\\
v_t + u v_x + v v_y + p_y &= 0,\label{NS1}\\
u_x + v_y &=0\nonumber. 
\end{align}
In a more compact formulation on matrix-vector form \cite{nordstrom2019, nordstrom2020spatial, Lauren2021} we have
\begin{align}
\tilde I U_t + A U_x + B U_y
= 0 \label{NS}
\end{align}
where $U=(u,v,p)^T$ and
\begin{equation}
\tilde I =
\begin{bmatrix}
1 & 0 & 0 \\
0 & 1 & 0 \\
0 & 0 & 0
\end{bmatrix},
\quad
A=
\begin{bmatrix}
u & 0 & 1 \\
0 & u & 0 \\
1 & 0 & 0
\end{bmatrix},
\quad
B=
\begin{bmatrix}
v & 0 & 0 \\
0 & v & 1 \\
0 & 1 & 0
\end{bmatrix}.
\label{ABI}
\end{equation}
To arrive at the proper skew-symmetric form, we rewrite (\ref{NS}) using the splitting technique  \cite{nordstrom2006conservative} as
\begin{equation}
\begin{array}{ll}
A U_x = \dfrac{1}{2}  \lbrack (A U)_x + A U_x - A_x U  \rbrack ,
\quad
B U_y = \dfrac{1}{2} \lbrack  (B U)_y + B U_y - B_y U  \rbrack.
\end{array}
\label{splitting_terms}
\end{equation}
By inserting (\ref{splitting_terms}) into (\ref{NS}) 
and recalling that we are dealing with an incompressible fluid, i.e., 
$
A_x + B_y = (u_x + v_y)  \tilde I =0,
$
we obtain the final form of the incompressible 2D Euler equations
\begin{equation}
\tilde I U_t 
+ \frac{1}{2}\left[ (A U)_x + A U_x 
+  (B U)_y + B U_y \right]
= 0.
\label{NS_splitting}
\end{equation} 
To connect with the general formulation (\ref{1Dprimal}): $P=\tilde I,  A_1=A/2,  A_2=B/2, C=0$. Since the matrices $A,B$ are already symmetric in the original formulation, the formulation (\ref{NS_splitting}) is in the proper skew-symmetric form, and the preceding theory in Sections \ref{sec:energyStab}, \ref{sec:newlinear} and \ref{sec:dual} can be applied. \begin{remark}
We obtain an estimate in the semi-norm  $\|U\|^2_{\tilde I}= \int_{\Omega} U^T \tilde I  U dxdy$ involving only the velocities.
\end{remark}
\begin{remark}
The divergence relation can also be used to formulate (\ref{NS}) in strictly conservative form.
\end{remark}

\subsection{The 3D incompressible Euler equations in cylindrical coordinates}\label{Euler3Dex}
A slightly more complicated example is given by the 3D incompressible Euler equations in cylindrical coordinates \cite{landau2006}
\begin{equation*}
   \begin{aligned}
   u_t + u \frac{\partial u}{\partial r}
       + \frac{v}{r} \frac{\partial u}{\partial \theta}
       + w \frac{\partial u}{\partial z} 
       + \frac{\partial p}{\partial r} 
        - \frac{v^{2}}{r}& = 0 
   \\
   v_t + u \frac{\partial v}{\partial r}
       + \frac{v}{r} \frac{\partial v}{\partial \theta}
       + w \frac{\partial v}{\partial z} 
       + \frac{1}{r}\frac{\partial p}{\partial \theta} 
        + \frac{uv}{r} & = 0 
   \\
   w_t + u \frac{\partial w}{\partial r}
       + \frac{v}{r} \frac{\partial w}{\partial \theta}
       + w \frac{\partial w}{\partial z} 
       + \frac{\partial p}{\partial z} & = 0 
   \\
       \frac{1}{r} \frac{\partial (ru)}{\partial r}
       + \frac{1}{r} \frac{\partial v}{\partial \theta}
       + \frac{\partial w}{\partial z} & = 0,
   \end{aligned}
\end{equation*}
or in a compact (almost) conservative formulation with subscripts denoting differentiation
\begin{equation}
\label{eq:cylinder_euler}
\begin{aligned}
   \tilde{I} U_t + 
   \frac{1}{r}
   \left[
      \left(r A U \right)_r +
   \left(B U \right)_\theta +
   \left(r C U \right)_z +
   R
   \right]
   = 0
	\, .
\end{aligned}
\end{equation}
In \eqref{eq:cylinder_euler}, $\tilde{I} = \text{diag}(1,1,1,0)$, $U = (u,v,w,p)^\top$ is the solution vector, where $u,v,w$ are the velocities in the $r$, $\theta$, $z$ direction, respectively, and $p$ is the pressure (divided by the constant density). Furthermore, we have
\begin{equation*}
 	A = \begin{pmatrix}
         u & 0 & 0 & 1
	      \\
         0 & u & 0 & 0
         \\
         0 & 0 & u & 0
	      \\
         1 & 0 & 0 & 0
      \end{pmatrix},
\quad
 	B = \begin{pmatrix}
         v & 0 & 0 & 0
	      \\
         0 & v & 0 & 1
         \\
         0 & 0 & v & 0
	      \\
         0 & 1 & 0 & 0
      \end{pmatrix},
\quad
 	C = \begin{pmatrix}
         w & 0 & 0 & 0
	      \\
         0 & w & 0 & 0
         \\
         0 & 0 & w & 1
	      \\
         0 & 0 & 1 & 0
      \end{pmatrix},
\quad
   R= \begin{pmatrix}
         -v^2-p  \\
         uv    \\
         0       \\
         0
      \end{pmatrix}\, .
\end{equation*}

Striving to use Propostion \ref{lemma:Matrixrelation}, we need the skew-symmetric form of \eqref{eq:cylinder_euler}. Again using the splitting technique  \cite{nordstrom2006conservative} we rewrite the conservative derivative formulation using the generic formula
\[
   (E U)_x = \frac{1}{2}\left[(EU)_x + E_x U + E U_x \right], \quad x= r, \theta, z \quad E = r A, B, r C
\]
and insert them into \eqref{eq:cylinder_euler}. The final result is
\begin{equation}
\label{eq:cylinder_euler_compact}
\begin{aligned}
  \tilde{I} U_t +   \frac{1}{2r}\left[ (r A U)_r + (rA) U_r +
                     (B U)_\theta + B U_\theta +
                     (rC U)_z + (rC) U_z +2 D U \right]=0,
\end{aligned}
\end{equation}
where the derivatives on the matrices and $R(U)$ are grouped together to form
\begin{equation}
\label{zeroorder}
  \frac{1}{2}((rA)_r + B_\theta + (rC)_z)U + R
   = D U, \quad \text{where}  \quad 
   D =
   \begin{pmatrix}
      0   & - v  & 0 & -1/2 \\
      v   & 0    & 0 &  0 \\
      0   & 0    & 0 &  0 \\
      1/2 & 0    & 0 &  0 
   \end{pmatrix}.
\end{equation}
To connect with the general formulation (\ref{1Dprimal}): multiply (\ref{eq:cylinder_euler_compact}) with $r$ and obtain $P=r \tilde I,  A_1=r A/2 , A_2=B/2, A_3=rC/2$ while the zero order term $C$ in (\ref{1Dprimal}) corresponds to $D$ in (\ref{zeroorder}).
Since the matrices in (\ref{eq:cylinder_euler_compact}) are symmetric (except $D$ which is skew-symmetric as required), Proposition \ref{lemma:Matrixrelation} applies.
\begin{remark}
The estimate is obtained in a semi-norm, in this case  $\|U\|^2_{\tilde I}= \int_{\Omega} U^T \tilde I  U r dr d\theta dz$. 
\end{remark}

\subsection{The 2D shallow water equations}\label{SWEex_2D}

The 2D shallow water equations (SWE) in primitive form \cite{whitham1974} excluding the influence of bottom topography are
\begin{equation}\label{eq:swNoncons_2D}
\begin{aligned}
    \phi_t + u\phi_x + v\phi_y + \phi (u_x + v_y) &=0,\\[0.1cm]
    u_t + u u_x + v u_y + \phi_x - f v &=0,\\[0.1cm]
    v_t + u v_x + v v_y + \phi_y + f u &=0,
\end{aligned}
\end{equation}
where $\phi = g h$ \cite{oliger1978} is the geopontential,  $h$ is the water height, $g$ is the gravitational constant and $u$ is the fluid velocity.  The Coriolis forces are included with the function $f$ which is typically a function of latitude \cite{shallowwaterbook,whitham1974}. Note that $h>0$ and $\phi>0$ from physical considerations. 
 
In both Burgers' equation and the two forms of incompressible Euler equations discussed above, transformation of variables was not required in order to obtain the proper skew-symmetric form. However, for the SWE this must be done. The total energy $\epsilon$ is the sum of the kinetic and potential energy. By multiplying the total energy with the gravitational constant $g$ we find
\begin{equation}\label{eq:totEng}
g \epsilon = \frac{1}{2}\left(\phi (u^2+v^2)+\phi^2\right).
\end{equation}
This is an auxiliary conserved positive quantity for smooth solutions of \eqref{eq:swNoncons_2D} from which we choose our new variables $U=(U_1,U_2,U_3)^T=(\phi, \sqrt{\phi} u, \sqrt{\phi} v))^T$ (each component now have the same physical units). To get an energy estimate we need governing equations for $U$. Repeated use of the chain rule in (\ref{eq:swNoncons_2D}) yields
\begin{equation}\label{eq:swNoncons_new_2D}
    U_t +  \mathcal{A} U_x + \mathcal{B} U_y+ \mathcal{C} U= 0,
\end{equation}
where
\begin{equation}\label{eq:swNoncons_new_matrix_A}
    \mathcal{A} = \begin{bmatrix}
     \frac{1}{2} u                                     & \sqrt{\phi}     & 0  \\
     \sqrt{\phi} -  \frac{u^2}{4 \sqrt{\phi}} & \frac{3}{2} u & 0  \\
                      - \frac{u v}{4 \sqrt{\phi}}  & \frac{1}{2} u & u 
       \end{bmatrix}=
       \begin{bmatrix}
     \frac{1}{2} \frac{U_2}{\sqrt{U_1}}                       & \sqrt{U_1}                                     & 0                                 \\
     \sqrt{U_1} -  \frac{U_2^2}{4 (\sqrt{U_1})^3 }     & \frac{3}{2} \frac{U_2}{\sqrt{U_1}} & 0                                  \\
                      -  \frac{U_2 U_3 }{4 (\sqrt{U_1})^3} & \frac{1}{2} \frac{U_3}{\sqrt{U_1}} & \frac{U_2}{\sqrt{U_1}}  
       \end{bmatrix}, 
\end{equation}       
\begin{equation}\label{eq:swNoncons_new_matrix_B}       
  \mathcal{B} = \begin{bmatrix}
     \frac{1}{2} v                                         & 0  & \sqrt{\phi}      \\
     - \frac{u v}{4 \sqrt{\phi}}                      &  v  &  \frac{1}{2} v  \\
       \sqrt{\phi} -  \frac{u^2}{4 \sqrt{\phi}}  & 0   & \frac{3}{2} v
       \end{bmatrix}=
       \begin{bmatrix}
     \frac{1}{2} \frac{U_3}{\sqrt{U_1}}                       & 0                                     &  \sqrt{U_1}                    \\
      -  \frac{U_2 U_3 }{4 (\sqrt{U_1})^3 }    & \frac{U_3}{\sqrt{U_1}} &  \frac{1}{2} \frac{U_2}{\sqrt{U_1}}                                \\
       \sqrt{U_1} -  \frac{U_3^2}{4 (\sqrt{U_1})^3} &0 & \frac{3}{2}  \frac{U_3}{\sqrt{U_1}}  \\
       \end{bmatrix},  \quad
  \mathcal{C} = \begin{bmatrix}
     0 & 0  & 0      \\
     0 & 0  &  -f    \\
     0 & +f & 0
       \end{bmatrix}.
\end{equation}

Equation (\ref{eq:swNoncons_new_2D}) is not in the skew-symmetric form (\ref{1Dprimal}) required for an estimate. To derive the skew-symmetric form, we first note that the different coordinate directions can be treated separately. We start with the $x$-direction and make the ansatz 
\begin{equation}\label{eq:swNoncons_new_matrix_ansatz_2D}
\mathcal{A} U_x =(A_1 U)_x+A_1^T U_x \quad \text{where}  \quad
A_1 = \begin{bmatrix}
     \alpha  u                 & \beta \sqrt{\phi}     & 0  \\
     \gamma \sqrt{\phi} & \theta u                  & 0  \\
     0                            & 0                            & \psi u  
       \end{bmatrix}=
       \begin{bmatrix}
      \alpha \frac{U_2}{\sqrt{U_1}}                       & \beta  \sqrt{U_1}                             & 0                                 \\
      \gamma \sqrt{U_1}                                      & \theta \frac{U_2}{\sqrt{U_1}}            & 0                                  \\
      0                                                                  & 0                                                      &  \psi \frac{U_2}{\sqrt{U_1}}  
       \end{bmatrix}, 
   \end{equation}
and  $ \alpha,\beta,\gamma,\theta,\psi$ are parameters to be determined. Equality in (\ref{eq:swNoncons_new_matrix_ansatz_2D}) leads to the one-parameter solution 
\begin{equation}\label{eq:swNoncons_new_matrix_ansatz_sol_A}
A_1 = \begin{bmatrix}
     \alpha u                 & (1-3 \alpha) \sqrt{\phi}  & 0 \\
     2\alpha  \sqrt{\phi} & \frac{1}{2} u                 &  0  \\
     0                            & 0                                  &\frac{1}{2} u  
       \end{bmatrix}=
       \begin{bmatrix}
     \alpha \frac{U_2}{\sqrt{U_1}}  & (1-3 \alpha) \sqrt{U_1}                  & 0 \\
     2\alpha  \sqrt{U_1}                 & \frac{1}{2}  \frac{U_2}{\sqrt{U_1}} &  0  \\
     0                                            & 0                                                    &\frac{1}{2} \frac{U_2}{\sqrt{U_1}}
       \end{bmatrix}.
\end{equation}
\begin{remark}
\label{Tomas}
The 1D skew-symmetric version of (\ref{eq:swNoncons_new_matrix_ansatz_sol_A}) was first derived in \cite{tomas}.
\end{remark}
Based on the result (\ref{eq:swNoncons_new_matrix_ansatz_sol_A})  in the $x$-direction and the structure of the original matrices (\ref{eq:swNoncons_new_matrix_A}) and (\ref{eq:swNoncons_new_matrix_B}) we make a guess (and verify) that the matrix $A_2$ in the $y$-direction is
\begin{equation}\label{eq:swNoncons_new_matrix_ansatz_sol_B}
A_2 = \begin{bmatrix}
     \beta v                 & 0 & (1-3 \beta) \sqrt{\phi}  \\
      0 & \frac{1}{2} v               &  0  \\
     2\beta \sqrt{\phi}                  & 0                                  &\frac{1}{2} v 
       \end{bmatrix}=
       \begin{bmatrix}
     \beta \frac{U_3}{\sqrt{U_1}}  & 0               &  (1-3 \beta) \sqrt{U_1}    \\
     0              & \frac{1}{2}  \frac{U_3}{\sqrt{U_1}} &  0  \\
     2\beta  \sqrt{U_1}                                         & 0                                                    &\frac{1}{2} \frac{U_3}{\sqrt{U_1}}
       \end{bmatrix}.
\end{equation}
The skew-symmetric matrix $\mathcal{C}$  now denoted $C$ in (\ref{eq:swNoncons_new_matrix_B}) remain the same after the transformation.

The 2D SWE are now transformed to
\begin{equation}\label{eq:swNoncons_new_skew}
    U_t +  (A_1 U)_x + A_1^T U_x+ (A_2 U)_y + A_2^T U_y+CU= 0,
\end{equation}
which has the required skew-symmetric form in Proposition \ref{lemma:Matrixrelation}. The rest of the theory presented in Sections \ref{sec:energyStab}, \ref{sec:newlinear} and \ref{sec:dual}  follows provided that the correct number and type of boundary conditions are given.
\begin{remark}
\label{rem:noalpha_2d}
The  energy rate cannot depend on the free parameters $\alpha$ and $\beta$ in the matrices  $A$ and $B$ since they are not present in (\ref{eq:swNoncons_new_2D}), (\ref{eq:swNoncons_new_matrix_A}) and (\ref{eq:swNoncons_new_matrix_B}). Hence as a sanity check we compute the boundary contraction
\begin{equation}\label{boundarmatrix_2D}
U^T (n_i A_i)U = 
U^T
       \begin{bmatrix}
      \frac{\alpha+\beta}{2} u_n                      &  \frac{1-\alpha}{2} n_1 \sqrt{U_1} &   \frac{1-\beta}{2} n_2 \sqrt{U_1}     \\
     \frac{1-\alpha}{2} n_1 \sqrt{U_1} & \frac{1}{2} u_n                             &    0                                                   \\
     \frac{1-\beta}{2} n_2 \sqrt{U_1}   & 0                                                  &\frac{1}{2} u_n
       \end{bmatrix}
U=u_n (U_1^2+\frac{1}{2} (U_2^2+U_3^2)),
\end{equation}
where for compactness we  introduced the normal velocity $u_n=n_1u+n_2v$. The dependency on the free parameters $\alpha$ and $\beta$ vanishes. Foregoing the analysis on boundary conditions in the next section, we remark that these parameters in the boundary contraction do not vanish in the linearised case.
\end{remark}
\begin{remark}
\label{compare_previous}
The estimate (\ref{eq:boundaryPart1}) with the same energy norm $\|U\|^2_P$,  the same boundary term (\ref{boundarmatrix_2D}) and cancelled volume terms was also obtained in  \cite{nordstrom2021linear}. In that case the original equations (\ref{eq:swNoncons_2D}), the variables $U=(\phi, u, v)$ and the norm matrix $P=diag(1, \phi, \phi)$ was used. The time and space dependent norm matrix was obtained from a symmetrisation requirement, which the formulation (\ref{eq:swNoncons_new_skew}) bypasses (such that $P=I_3$ can be used). This simplifies the production of a stable discrete approximation as discussed in Section \ref{numerics}.
\end{remark}

\section{Analysis and discussion of boundary conditions for nonlinear and linearised problems}\label{boundary_conditions}
A nonlinear and linear analysis may lead to a different number and type of boundary conditions required for an energy bound. This was shown in \cite{nordstrom2021linear} where the boundary coefficient matrix $n_i A_i$ were found to be distinctly different in the linear and nonlinear case. There are also cases where this discrepancy do not occur, see \cite{nordstrom2019, nordstrom2020spatial, Lauren2021,Nordstrom2007874} for examples. We will discuss this issue in light of the examples and novel formulations discussed above.

Comparing the nonlinear skew-symmetric problem with the linearised one, i.e.
\begin{equation*}
P U_t + (A_i(U) U)_{x_i}+A^T_i(U)U_{x_i}+C(U) U=0 \quad \text{with }  \quad P U^\prime_t +(A_i(\bar U) U^\prime)_{x_i}+A^T_i(\bar U) U^\prime_{x_i}+C(\bar U) U^\prime=0
\end{equation*}
we find that they differ only in the arguments of the matrices. There are no terms missing, and hence the combined boundary coefficient matrix $n_iA_i$ has the same structure in both cases. This suggests that the same procedure for deriving boundary conditions can be used in both the nonlinear and linearised case. In other words, the discrepancy when it comes to boundary conditions seen in \cite{nordstrom2021linear} seems to vanish with the new linearisation formulation. This is, in fact, also the situation for the Burgers' equation and the two forms incompressible Euler equations discussed above.

However, when considering the 2D SWE, the situation becomes more complicated due to the presence of the free parameters $\alpha$ and $\beta$. As shown in Remark \ref{rem:noalpha_2d}, the dependence of $\alpha$ and $\beta $ in the nonlinear boundary term $U^T (n_iA_i(U))  U$ is an illusion and cancel after contraction. However, the linear case is different because two sets of different variables $\bar U$ and $U^\prime$ are involved, and the boundary term
$(U^\prime)^T (n_iA_i(\bar U)) U^\prime$ retain the dependence on $\alpha$ and $\beta$. From (\ref{boundarmatrix_2D}) we found the parameter independent nonlinear boundary term to be
\begin{equation}
\label{eq:nonlin_compare}
U^T (n_iA_i(U)) U=u_n (U_1^2+\frac{1}{2} (U_2^2+U_3^2))= 
U^T
       \begin{bmatrix}
     u_n                      &  0                                                 &  0                      \\
     0                          & \frac{1}{2} u_n                             &    0                    \\
     0                          & 0                                                   &\frac{1}{2} u_n
      \end{bmatrix}
U
\end{equation}
while the parameter dependent linearised version is
\begin{equation}
\label{eq:lin_compare}
(U^\prime)^T (n_iA_i(\bar U)) U^\prime = 
(U^\prime)^T
       \begin{bmatrix}
     \frac{\alpha+\beta}{2}  \bar u_n                      &  \frac{1-\alpha}{2} n_1 \sqrt{\bar U_1} &   \frac{1-\beta}{2} n_2 \sqrt{\bar U_1}     \\
     \frac{1-\alpha}{2} n_1 \sqrt{\bar U_1} & \frac{1}{2} \bar u_n                             &    0                                                   \\
     \frac{1-\beta}{2} n_2 \sqrt{\bar U_1}   & 0                                                  &\frac{1}{2} \bar u_n
       \end{bmatrix}
U^\prime.
\end{equation}
The formulation  (\ref{eq:lin_compare}) holds for all $\alpha, \beta$. 

To shed some light on the difference between the linear and nonlinear boundary treatment, consider the special choice $\alpha=\beta=1$ which produces a boundary term similar to the nonlinear one 
\begin{equation}
\label{eq:lin_compare_0}
(U^\prime)^T (n_iA_i(\bar U)) U^\prime =
(U^\prime)^T
       \begin{bmatrix}
     \bar u_n                      &  0                                                 &  0                      \\
     0                          & \frac{1}{2} \bar u_n                             &    0                    \\
     0                          & 0                                                   &\frac{1}{2} \bar u_n
      \end{bmatrix}
U^\prime.
\end{equation}
Let us now compare (\ref{eq:nonlin_compare}) and (\ref{eq:lin_compare_0}). For inflow, when both $u_n$  and $\bar u_n$ are negative, three boundary conditions seem to be needed in both cases while at outflow when both $\bar u_n$ and $u_n$ are positive, none is required. 

However, there is a significant difference between (\ref{eq:nonlin_compare}) and (\ref{eq:lin_compare_0}) since the diagonal entries in (\ref{eq:nonlin_compare}) are functions of the solution 
($u_n=(n_1 U_2 + n_2 U_3)/\sqrt{U_1} $) while they can be considered as external data in the linear case as $\bar u_n$ is independent of $U^\prime$. 
Following \cite{nordstrom2021linear} we aim for a minimal number of inflow boundary conditions and rewrite the nonlinear boundary term as
\begin{equation}
\label{eq:nonlin_compare_v1}
U^T (n_iA_i(U)) U=\frac{1}{2 U_n \sqrt{U_1}}
\begin{bmatrix}
U_1^2                  \\
U_n^2 + U_1^2  \\
U_n U_{\tau}
\end{bmatrix}^T
\begin{bmatrix}
     -1                   &  0                                                 &  0                      \\
     0                          & 1                            &    0                    \\
     0                          & 0                                                   & 1
 \end{bmatrix}
 \begin{bmatrix}
U_1^2                  \\
U_n^2 + U_1^2  \\
U_n U_{\tau}
\end{bmatrix},
\end{equation}
where we introduced the normal $U_n=n_1 U_2 + n_2 U_3=u_n \sqrt{\phi} $ and tangential $U_{\tau}=-n_2 U_2 + n_1 U_3=u_{\tau}  \sqrt{\phi} $ scaled velocities. Considering non-glancing boundaries we have  $\min{\lvert U_n \rvert}  \geq \delta_n >0$  and since $\phi>0$ we also have $\min{\lvert \sqrt{U_1} \rvert}  \geq \delta_1 >0$. This leads to a bound using only two boundary conditions (recall that $U_n < 0$) instead of three. By specifying $U_n^2 + U_1^2 = g_2^2$ and $U_n U_{\tau}=g_3^2$ where $g_2, g_3$  are bounded functions, we get
\begin{equation}\label{eq:nonlinest}
\frac{1}{2} \frac{d}{dt}\|U\|^2 =- \oint\limits_{\partial\Omega} U^T (n_iA_i(U)) U ds   \leq \oint\limits_{\partial\Omega}  \frac{-U_1^4 + g_2^4 + g_3^4}{2 \min{\lvert U_n \rvert} \min{\lvert \sqrt{U_1} \rvert} }ds <  \infty. 
\end{equation}
The reformulation of the boundary term in the nonlinear case from (\ref{eq:nonlin_compare}) to (\ref{eq:nonlin_compare_v1}) was possible only because the diagonal entries in the boundary matrix were functions of the solution. In the linearised case, no such reformulation can be done. 
\begin{remark}
In the nonlinear outflow case, the rewritten boundary term in (\ref{eq:nonlin_compare_v1}) indicates that one boundary condition is required. However, this is an illusion. Aiming for a minimal number, the relation (\ref{eq:nonlin_compare}) shows that no boundary condition is required, precisely as in the linear case.
\end{remark}
\begin{remark}
An energy estimate can always be obtained by specifying too many boundary conditions \cite{nordstrom_roadmap,nordstrom2020}. But, in general overprescribing the boundary conditions means that existence cannot be obtained for linear problems. The precise existence requirement is not known for nonlinear problems, but it is reasonable to assume (unless proven otherwise) that similar requirements as for linear problems hold.
\end{remark}
\begin{remark}
\label{rem:obs}
The free parameters $\alpha$ and $\beta$ in the 2D SWE appeared since we had to derive new evolution equations for the transformed variables forming the energy. New equations were not required for Burgers' equation and the incompressible Euler equations since they already had the required form in the original variables. The number of boundary conditions for the linear 2D SWE depend upon the parameters $\alpha$ and $\beta$ which may or may not match the number required in the nonlinear case. One may speculate that nonlinear equations written on skew-symmetric form with the original variables have the same number and form of boundary conditions as the linearised version. However, if a transformation is required, caution must be taken since free parameters may appear (implicitly or explicitly), rendering the linear analysis unreliable for the nonlinear problem (or vice versa). This situation may have bearing on the recent results in \cite{gassner2020stability,ranocha2021}.
\end{remark}

\section{A stable and energy conserving numerical approximation}\label{numerics}
To exemplify the straightforward construction of stable and energy conserving linear and nonlinear schemes based on the skew-symmetric formulation in Proposition \ref{lemma:Matrixrelation}, we consider a summation-by-parts (SBP) approximation of the following general 2D problem
\begin{equation}\label{2DGEN_Cont}
P U_t+(A_1 U)_x+A_1^T U_x + (A_2 U)_y+A_2^T U_y=0.
\end{equation}
The boundary conditions (imposed through SAT terms or numerical flux functions) are assumed to be dissipative and ignored (we focus on the energy conserving properties of the numerical scheme). Equation (\ref{2DGEN_Cont}) is semi-discretised in space using summation-by-parts operators as
\begin{equation}\label{SWE_Disc}
{\bf P}  \vec U_t+{\bf D_x} ({\bf A_1}  \vec U)+{\bf A_1^T} {\bf D_x}  (\vec U) +{\bf D_y} ({\bf A_2}  \vec U)+{\bf A_2^T} {\bf D_y}  (\vec U)=0,
\end{equation}
where ${\bf P} = P \otimes I_x   \otimes I_y$ and $\vec U=(\vec U_1^T, \vec U_2^T,...,\vec U_n^T)^T$ include approximations of  $U=(U_1,U_2,...,U_n)^T$ in each node.  The matrix elements of ${\bf A_1},{\bf A_2}$ are matrices with node values of the matrix elements in $A_1,A_2$ injected on the diagonal as exemplified below
\begin{equation}
\label{illustration}
A_1 =
\begin{pmatrix}
      a_{11}   &  \ldots  & a_{1n} \\
       \vdots   & \ddots & \vdots \\
       a_{n1} &  \ldots  & a_{nn}
\end{pmatrix}, \quad
{\bf A_1} =
\begin{pmatrix}
      {\bf a_{11}}   &  \ldots  &  {\bf a_{1n}}  \\
       \vdots          & \ddots &  \vdots           \\
        {\bf a_{n1}} &  \ldots &  {\bf a_{nn}} 
\end{pmatrix}, \quad
{\bf a_{ij}} =diag(a_{ij}(x_1,y_1), \ldots, a_{ij}(x_N,y_M)).
\end{equation}

Moreover ${\bf D_x}=I_n \otimes D_{x} \otimes I_y$ and ${\bf D_y}=I_n \otimes I_x \otimes D_y$ where $D_{x,y}=P^{-1}_{x,y}Q_{x,y}$ are 1D SBP difference operators, $P_{x,y}$ are positive definite diagonal quadrature matrices, $Q_{x,y}$ satisfies the SBP constraint $Q_{x,y}+Q_{x,y}^T=B_{x,y}=diag[-1,0,...,0,1]$, $\otimes$ denotes the Kronecker product  and $I$ with subscripts denote identity matrices. All matrices have appropriate sizes such that the matrix-matrix and matrix-vector operations are defined. 
Based on the 1D SBP operators, the 2D SBP relations mimicking multi-dimensional integration by parts are given by 
\begin{equation}\label{Multi-SBP}
\vec U^T  \tilde {\bf P} {\bf D_x} \vec V= -({\bf D_x} \vec U)^T \tilde {\bf P} \vec V + \vec U^T {\bf B_x} \vec V, \quad 
\vec U^T \tilde {\bf P} {\bf D_y} \vec V= -({\bf D_y} \vec U)^T  \tilde {\bf P} \vec V + \vec U^T {\bf B_y} \vec V,
\end{equation}
where  $\vec U^T {\bf B_x} \vec V$ and $\vec U^T {\bf B_y} \vec V$ contain numerical integration along rectangular domain boundaries. In (\ref{Multi-SBP}) we have used
$ \tilde {\bf P}=I_n  \otimes P_x   \otimes P_y$, ${\bf B_x}=(I_n \otimes B_x \otimes P_y)$ and ${\bf B_y}=(I_n \otimes P_x \otimes B_y)$. 

The discrete energy method (multiply (\ref{SWE_Disc}) from the left with  $\vec U^T  \tilde {\bf P}$) yields 
\begin{equation}\label{SWE_Disc_energy}
 \vec U^T ({\bf P} \tilde {\bf P}) \vec U_t+ \vec U^T  \tilde {\bf P} {\bf D_x} ({\bf A_1}  \vec U)+ \vec U^T \tilde {\bf P} {\bf A_1^T}{\bf D_x}  (\vec U) + 
                                                              \vec U^T  \tilde {\bf P} {\bf D_y} ({\bf A_2}  \vec U)+ \vec U^T \tilde {\bf P} {\bf A_2^T} {\bf D_y}  (\vec U) =0.
\end{equation}
Next we introduce the notation  $\| \vec U\|_{{\bf P} \tilde {\bf P}}^2=\vec U^T ({\bf P} \tilde {\bf P}) \vec U$, apply the SBP relations (\ref{Multi-SBP}) and arrive at
 \begin{equation}\label{Disc_energy_final}
 \begin{aligned}
\frac{1}{2} \dfrac{d}{dt} \| \vec U\|_{{\bf P}  \tilde {\bf P}}^2 + \vec U^T {\bf B_x} ({\bf A_1}  \vec U) + \vec U^T {\bf B_y} ({\bf A_2}  \vec U)
&=({\bf D_x}  \vec U)^T \tilde {\bf P}( {\bf A_1}\vec U)-\vec U^T \tilde {\bf P} {\bf A_1^T}({\bf D_x}  \vec U) \\
&+({\bf D_y}  \vec U)^T \tilde {\bf P} ({\bf A_2}\vec U)-\vec U^T \tilde {\bf P} {\bf A_2^T} ({\bf D_x}  \vec U)
\end{aligned}
\end{equation}
which mimics (\ref{eq:boundaryPart1}) perfectly (ignoring the zero order term). By using the fact that  $\tilde {\bf P} {\bf A_1}={\bf A_1} \tilde {\bf P}$ and  $\tilde {\bf P} {\bf A_2}={\bf A_2}\tilde {\bf P}$ (since the matrices involved consist of diagonal blocks), the right hand side of (\ref{Disc_energy_final}) vanishes. Hence, the energy rate depends only on boundary terms and the scheme is energy conserving in the semi-discrete setting. 
\begin{remark}
An energy bound and stability can be obtained by adding proper weak dissipative boundary conditions using the SAT technique or numerical flux functions.
\end{remark}
\begin{remark}
It is irrelevant whether the matrices $\bf A_1$ and $\bf A_2$ are functions of the solution or not, i.e., if the problem is linear or nonlinear. The skew-symmetric formulation, a summation-by-parts discretisation and a proper weak boundary treatment are all that matters for stability and energy conservation.
\end{remark}
\begin{remark}
The energy conservation for the linear and nonlinear primal problem illustrated above, holds also for the corresponding linear and nonlinear dual problems.
\end{remark}

\section{Summary and conclusions}\label{sec:conclusion}
The standard linearisation procedure often results in a confusing contradiction where the nonlinear problem conserves energy and has an energy bound but the linearised variable coefficient version does not.  We have shown that a specific skew-symmetric form of the primal nonlinear problem leads to an energy bound and energy conservation. Next, it was shown that this skew-symmetric form together with a non-standard linearisation procedure lead to an energy bounded and energy conserving formulation also of the new slightly modified linearised variable coefficient problem. We also showed that the corresponding linear and nonlinear dual (or self-adjoint) problems retain these properties due to this specific formulation. 

The scalar Burgers' equation, the incompressible 2D Euler equations, the incompressible 3D Euler equations in cylindrical coordinates and the 2D shallow water equations were analysed to illustrate the new theory. From a study of these examples, we tentatively found that nonlinear equations on skew-symmetric form in the original variables have the same number and form of boundary conditions as the linearised version. However, if a variable transformation was required, caution must be taken, and the linear analysis may not be trustworthy. We concluded the paper by connecting the continuous analysis to a semi-discrete approximation. We showed that the skew-symmetric formulation automatically produced energy conserving and energy stable and numerical schemes for both linear and nonlinear primal and corresponding dual problems if these are formulated on summation-by-parts form. 

The main finding in this paper can be summarised as follows. To obtain a formulation leading to an energy bound and energy conservation for general nonlinear hyperbolic IBVP, a number of actions are in general required. Firstly, one must choose dependent variables that forms an appropriate energy norm. Secondly, evolution equations for the new variables on skew-symmetric form must be derived.  
Finally, one must derive boundary conditions that limit the boundary terms. Some of the steps mentioned above can often be bypassed. In the examples treated, neither the Burgers' nor the incompressible Euler equations needed new variables and the skew-symmetric form followed directly. For a general nonlinear IBVP, choosing a set of transformation variables to produce a skew-symmetric formulation is likely the most difficult task.
\section*{Acknowledgments}

Many thanks to my colleagues Fredrik Laur{\'e}n, Tomas Lundquist and Andrew R. Winters for helpful and productive comments on the manuscript. Jan Nordstr\"{o}m was supported by Vetenskapsr{\aa}det, Sweden [award no.~2018-05084 VR] and the Swedish e-Science Research Center (SeRC).

\bibliographystyle{elsarticle-num}
\bibliography{References_andrew,References_Fredrik}

\begin{thebibliography}{10}
\expandafter\ifx\csname url\endcsname\relax
  \def\url#1{\texttt{#1}}\fi
\expandafter\ifx\csname urlprefix\endcsname\relax\def\urlprefix{URL }\fi
\expandafter\ifx\csname href\endcsname\relax
  \def\href#1#2{#2} \def\path#1{#1}\fi

\bibitem{kreiss1970}
H.-O. Kreiss, Initial boundary value problems for hyperbolic systems, Commun.
  Pur. Appl. Math. 23~(3) (1970) 277--298.

\bibitem{kreiss1989initial}
H.-O. Kreiss, J.~Lorenz, Initial-boundary value problems and the
  {N}avier-{S}tokes equations, Vol.~47, SIAM, 1989.

\bibitem{Gustafsson1978}
B.~Gustafsson, A.~Sundstrom, Incompletely parabolic problems in fluid dynamics,
  SIAM J. Appl. Math. 35~(2) (1978) 343--357.

\bibitem{gustafsson1995time}
B.~Gustafsson, H.-O. Kreiss, J.~Oliger, Time dependent problems and difference
  methods, Vol.~24, JWS, 1995.

\bibitem{oliger1978}
J.~Oliger, A.~Sundstr{\"o}m, Theoretical and practical aspects of some initial
  boundary value problems in fluid dynamics, SIAM J. Appl. Math. 35~(3) (1978)
  419--446.

\bibitem{nordstrom2020}
J.~Nordstr\"{o}m, T.~M. Hagstrom, The number of boundary conditions for initial
  boundary value problems, SIAM Journal on Numerical Analysis 58~(5) (2020)
  2818--2828.

\bibitem{nordstrom_roadmap}
J.~Nordstr\"{o}m, A roadmap to well posed and stable problems in computational
  physics, {J. Sci. Comput.} 71~(1) (2017) 365--385.

\bibitem{nordstrom2005}
J.~Nordstr\"{o}m, M.~Sv\"{a}rd, Well posed boundary conditions for the
  {N}avier--{S}tokes equations, SIAM J. Numer. Anal. 43 (2005) 1231--1255.

\bibitem{ghader2014}
S.~Ghader, J.~Nordstr\"{o}m, Revisiting well-posed boundary conditions for the
  shallow water equations, Dynam. Atmos. Oceans 66 (2014) 1--9.

\bibitem{nordstrom2019}
J.~Nordstr\"{o}m, C.~L. Cognata, Energy stable boundary conditions for the
  nonlinear incompressible {N}avier--{S}tokes equations, Math. Comput. 88~(316)
  (2019) 665--690.

\bibitem{nordstrom2020spatial}
J.~Nordstr{\"o}m, F.~Laur{\'e}n, The spatial operator in the incompressible
  {N}avier--{S}tokes, {O}seen and {S}tokes equations, Computer Methods in
  Applied Mechanics and Engineering 363 (2020) 112857.

\bibitem{Lauren2021}
F.~Lauren, J.~Nordstr{\"o}m, Spectral properties of the incompressible
  {N}avier--{S}tokes equations, Journal of Computational Physics 429 (2021)
  110019.

\bibitem{tadmor1984}
E.~Tadmor, Skew-selfadjoint form for systems of conservation laws, J. Math.
  Anal. Appl. 103~(2) (1984) 428--442.

\bibitem{Tadmor1987}
E.~Tadmor, The numerical viscosity of entropy stable schemes for systems of
  conservation laws, Math. Comput. 49~(179) (1987) 91--103.

\bibitem{Tadmor2003}
E.~Tadmor, Entropy stability theory for difference approximations of nonlinear
  conservation laws and related time-dependent problems, Acta Numer. 12 (2003)
  451--512.

\bibitem{godunov1961interesting}
S.~K. Godunov, An interesting class of quasilinear systems, in: Dokl. Acad.
  Nauk SSSR, Vol.~11, 1961, pp. 521--523.

\bibitem{volpert1967}
A.~I. Vol'pert, The space {BV} and quasilinear equations, Math. USSR SB+ 10
  (1967) 257--267.

\bibitem{kruzkov1970}
S.~N. Kru\v{z}kov, First order quasilinear equations in several independent
  variables, Math. USSR SB+ 10~(2) (1970) 127--243.

\bibitem{dafermos1973entropy}
C.~M. Dafermos, The entropy rate admissibility criterion for solutions of
  hyperbolic conservation laws, J. Differ. Equations 14~(2) (1973) 202--212.

\bibitem{lax1973}
P.~D. Lax, Hyperbolic systems of conservation laws and the mathematical theory
  of shock waves, in: CBMS Regional Conference Series in Applied Mathematics,
  Vol.~11, SIAM, 1973.

\bibitem{harten1983}
A.~Harten, On the symmetric form of systems of conservation laws with entropy,
  J. Comput. Phys. 49 (1983) 151--164.

\bibitem{dubois1988}
F.~Dubois, P.~LeFloch, Boundary conditions for nonlinear hyperbolic systems of
  conservation laws, Journal of Differential Equations 71~(1) (1988) 93--122.

\bibitem{hindenlang2019}
F.~Hindenlang, G.~Gassner, D.~Kopriva, Stability of wall boundary condition
  procedures for discontinuous {G}alerkin spectral element approximations of
  the compressible {E}uler equations, Lecture Notes in Computational Science
  and Engineering 134 (2020) 3--19.

\bibitem{parsani2015entropy}
M.~Parsani, M.~H. Carpenter, E.~J. Nielsen, Entropy stable wall boundary
  conditions for the three-dimensional compressible {N}avier--{S}tokes
  equations, J. Comput. Phys. 292 (2015) 88--113.

\bibitem{svard2012}
M.~Sv{\"a}rd, S.~Mishra, Entropy stable schemes for initial-boundary-value
  conservation laws, Z. Angew. Math. Phys. 63~(6) (2012) 985--1003.

\bibitem{svard2021entropy}
M.~Sv{\"a}rd, Entropy stable boundary conditions for the {E}uler equations,
  Journal of Computational Physics 426 (2021) 109947.

\bibitem{Jameson1988233}
A.~Jameson, Aerodynamic design via control theory, Journal of Scientific
  Computing 3~(3) (1988) 233--260.

\bibitem{Jameson1998213}
A.~Jameson, L.~Martinelli, N.~Pierce, Optimum aerodynamic design using the
  {N}avier-{S}tokes equations, Theoretical and Computational Fluid Dynamics
  10~(1-4) (1998) 213--237.

\bibitem{Giles2000393}
M.~Giles, N.~Pierce, An introduction to the adjoint approach to design, Flow,
  Turbulence and Combustion 65~(3-4) (2000) 393--415.

\bibitem{Giles2002145}
M.~Giles, E.~S{\"u}li, Adjoint methods for {PDE}s: A posteriori error analysis
  and postprocessing by duality, Acta Numerica 11 (2002) 145--236.

\bibitem{Nielsen20021155}
E.~Nielsen, W.~Anderson, Recent improvements in aerodynamic design optimization
  on unstructured meshes, AIAA Journal 40~(6) (2002) 1155--1163.

\bibitem{Fidkowski2011673}
K.~Fidkowski, D.~Darmofal, Review of output-based error estimation and mesh
  adaptation in computational fluid dynamics, AIAA Journal 49~(4) (2011)
  673--694.

\bibitem{gassner2020stability}
G.~J. Gassner, M.~Sv{\"a}rd, F.~J. Hindenlang, Stability issues of
  entropy-stable and/or split-form high-order schemes (2020).
\newblock \href {http://arxiv.org/abs/2007.09026} {\path{arXiv:2007.09026}}.

\bibitem{ranocha2021}
H.~Ranocha, G.~J. Gassner, Preventing pressure oscillations does not fix local
  linear stability issues of entropy-based split-form high-order schemes,
  Communications on Applied Mathematics and Computation (2021) 1--24.

\bibitem{berg2012}
J.~Berg, J.~Nordstr{\"o}m, Superconvergent functional output for time-dependent
  problems using finite differences on summation-by-parts form, Journal of
  Computational Physics 231~(20) (2012) 6846--6860.

\bibitem{berg2013}
J.~Berg, J.~Nordstr{\"o}m, On the impact of boundary conditions on dual
  consistent finite difference discretizations, Journal of Computational
  Physics 236~(1) (2013) 41--55.

\bibitem{berg2014}
J.~Berg, J.~Nordstr{\"o}m, Duality based boundary conditions and dual
  consistent finite difference discretizations of the {N}avier-{S}tokes and
  {E}uler equations, Journal of Computational Physics 259 (2014) 135--153.

\bibitem{nordstrom_dual_2017}
J.~Nordstr{\"o}m, F.~Ghasemi, On the relation between conservation and dual
  consistency for summation-by-parts schemes, Journal of Computational Physics
  344 (2017) 437--439.

\bibitem{Nordstrom-Ghasemi2020}
J.~Nordstr{\"o}m, F.~Ghasemi, The relation between primal and dual boundary
  conditions for hyperbolic systems of equations, Journal of Computational
  Physics 401 (2020).

\bibitem{Thalabard2020}
S.~Thalabard, J.~Bec, A.~Mailybaev, From the butterfly effect to spontaneous
  stochasticity in singular shear flows, Communications Physics 3~(1) (2020)
  673--694.

\bibitem{Lohner2014742}
R.~Lohner, D.~Britto, A.~Michailski, E.~Haug, Butterfly-effect for massively
  separated flows, Engineering Computations (Swansea, Wales) 31~(4) (2014)
  742--757.

\bibitem{Wang20131}
Q.~Wang, Forward and adjoint sensitivity computation of chaotic dynamical
  systems, Journal of Computational Physics 235 (2013) 1--13.

\bibitem{wilcox2013}
S.~Kaijima, R.~Bouffanais, K.~Willcox, S.~Naidu, Computational fluid dynamics
  for architectural design, Architectural Design 83~(2) (2013) 118--123.

\bibitem{Lorenz63}
E.~Lorenz, Deterministic nonperiodic flow, Journal of the Atmospheric Sciences
  20 (1963) 130--141.

\bibitem{Wang2014210}
Q.~Wang, R.~Hu, P.~Blonigan, Least squares shadowing sensitivity analysis of
  chaotic limit cycle oscillations, Journal of Computational Physics 267 (2014)
  210--224.

\bibitem{Wang2014156}
Q.~Wang, Convergence of the least squares shadowing method for computing
  derivative of ergodic averages, SIAM Journal on Numerical Analysis 52~(1)
  (2014) 156--170.

\bibitem{Lea2000523}
D.~Lea, M.~Allen, T.~Haine, Sensitivity analysis of the climate of a chaotic
  system, Tellus, Series A: Dynamic Meteorology and Oceanography 52~(5) (2000)
  523--532.

\bibitem{Eyink20041867}
G.~Eyink, T.~Haine, D.~Lea, Ruelle's linear response formula, ensemble adjoint
  schemes and {L}evy flights, Nonlinearity 17~(5) (2004) 1867--1889.

\bibitem{Thuburn200573}
J.~Thuburn, Climate sensitivities via a {F}okker-{P}lanck adjoint approach,
  Quarterly Journal of the Royal Meteorological Society 131~(605) (2005)
  73--92.

\bibitem{Blonigan2014660}
P.~Blonigan, Q.~Wang, Probability density adjoint for sensitivity analysis of
  the mean of chaos, Journal of Computational Physics 270 (2014) 660--686.

\bibitem{Ibragimov2006742}
N.~Ibragimov, Integrating factors, adjoint equations and {L}agrangians, Journal
  of Mathematical Analysis and Applications 318~(2) (2006) 742--757.

\bibitem{NEWIbragimov2011}
N.~Ibragimov, Nonlinear self-adjointness and conservation laws, Journal of
  Physics A: Mathematical and Theoretical 44~(43) (2011) 432002.

\bibitem{Ibragimov2007311}
N.~Ibragimov, A new conservation theorem, Journal of Mathematical Analysis and
  Applications 333~(1 SPEC. ISS.) (2007) 311--328.

\bibitem{S009630031200667420121001}
M.~Gandarias, M.~Bruzon, Conservation laws for a class of quasi self-adjoint
  third order equations., Applied Mathematics and Computation 219~(2) (2012)
  668 -- 678.

\bibitem{edselc.2-52.0-7995823472520110701}
M.~Gandarias, Weak self-adjoint differential equations., Journal of Physics A:
  Mathematical and Theoretical 44~(26) (2011) 262001.

\bibitem{OLDIbragimov2011}
N.~Ibragimov, M.~Torrisi, R.~Tracin, Self-adjointness and conservation laws of
  a generalized {B}urgers equation, Journal of Physics A: Mathematical and
  Theoretical 44~(14) (2011) 145201.

\bibitem{Zhang2013}
Z.-Y. Zhang, Approximate nonlinear self-adjointness and approximate
  conservation laws, Journal of Physics A: Mathematical and Theoretical 46~(15)
  (2013) 155203.

\bibitem{Tracina2015}
R.~Tracina, Nonlinear self-adjointness: A criterion for linearization of
  {PDE}s, Journal of Physics A: Mathematical and Theoretical 48~(6) (2015)
  06FT01.

\bibitem{Strang196437}
G.~Strang, Accurate partial difference methods - ii. non-linear problems,
  Numerische Mathematik 6~(1) (1964) 37--46.

\bibitem{nordstrom2021linear}
J.~Nordstr{\"o}m, A.~R. Winters, Linear and nonlinear analysis of the shallow
  water equations (2021).
\newblock \href {http://arxiv.org/abs/1907.10713} {\path{arXiv:1907.10713}}.

\bibitem{Nordstrom2007874}
J.~Nordstr\"om, K.~Mattsson, C.~Swanson, Boundary conditions for a divergence
  free velocity-pressure formulation of the {N}avier-{S}tokes equations,
  Journal of Computational Physics 225~(1) (2007) 874--890.

\bibitem{Arnold20011749}
D.~Arnold, F.~Brezzi, B.~Cockburn, L.~Donatella~Marini, Unified analysis of
  discontinuous {G}alerkin methods for elliptic problems, SIAM Journal on
  Numerical Analysis 39~(5) (2001) 1749--1779.

\bibitem{nordstrom2009stable}
J.~Nordstr{\"o}m, J.~Gong, E.~Van~der Weide, M.~Sv{\"a}rd, A stable and
  conservative high order multi-block method for the compressible
  {N}avier--{S}tokes equations, Journal of Computational Physics 228~(24)
  (2009) 9020--9035.

\bibitem{svard2007stable}
M.~Sv{\"a}rd, M.~H. Carpenter, J.~Nordstr{\"o}m, A stable high-order finite
  difference scheme for the compressible {N}avier--{S}tokes equations,
  far-field boundary conditions, Journal of Computational Physics 225~(1)
  (2007) 1020--1038.

\bibitem{svard2008stable}
M.~Sv{\"a}rd, J.~Nordstr{\"o}m, A stable high-order finite difference scheme
  for the compressible {N}avier--{S}tokes equations: No-slip wall boundary
  conditions, Journal of Computational Physics 227~(10) (2008) 4805--4824.

\bibitem{nordstrom2012weak}
J.~Nordstr{\"o}m, S.~Eriksson, P.~Eliasson, Weak and strong wall boundary
  procedures and convergence to steady-state of the {N}avier--{S}tokes
  equations, Journal of Computational Physics 231~(14) (2012) 4867--4884.

\bibitem{nordstrom2003finite}
J.~Nordstr{\"o}m, K.~Forsberg, C.~Adamsson, P.~Eliasson, Finite volume methods,
  unstructured meshes and strict stability for hyperbolic problems, Applied
  Numerical Mathematics 45~(4) (2003) 453--473.

\bibitem{carpenter2014entropy}
M.~H. Carpenter, T.~C. Fisher, E.~J. Nielsen, S.~H. Frankel, Entropy stable
  spectral collocation schemes for the {N}avier--{S}tokes equations:
  Discontinuous interfaces, SIAM Journal on Scientific Computing 36~(5) (2014)
  B835--B867.

\bibitem{carpenter1996spectral}
M.~H. Carpenter, D.~Gottlieb, Spectral methods on arbitrary grids, Journal of
  Computational Physics 129~(1) (1996) 74--86.

\bibitem{castonguay2013energy}
P.~Castonguay, D.~M. Williams, P.~E. Vincent, A.~Jameson, Energy stable flux
  reconstruction schemes for advection--diffusion problems, Computer Methods in
  Applied Mechanics and Engineering 267 (2013) 400--417.

\bibitem{huynh2007flux}
H.~T. Huynh, A flux reconstruction approach to high-order schemes including
  discontinuous {G}alerkin methods, in: 18th AIAA Computational Fluid Dynamics
  Conference, 2007, p. 4079.

\bibitem{gassner2013skew}
G.~J. Gassner, A skew-symmetric discontinuous {G}alerkin spectral element
  discretization and its relation to {SBP-SAT} finite difference methods, SIAM
  Journal on Scientific Computing 35~(3) (2013) A1233--A1253.

\bibitem{hesthaven1996stable}
J.~S. Hesthaven, D.~Gottlieb, A stable penalty method for the compressible
  {N}avier--{S}tokes equations: I. {O}pen boundary conditions, SIAM Journal on
  Scientific Computing 17~(3) (1996) 579--612.

\bibitem{kopriva2021}
D.~Kopriva, G.~Gassner, J.~Nordstr{\"o}m, Stability of discontinuous galerkin
  spectral element schemes for wave propagation when the coefficient matrices
  have jumps, Journal of Scientific Computing 88~(1) (2021) 3.

\bibitem{abgrall2020analysis}
R.~Abgrall, J.~Nordstr{\"o}m, P.~{\"O}ffner, S.~Tokareva, Analysis of the
  {SBP-SAT} stabilization for finite element methods part i: linear problems,
  Journal of Scientific Computing 85~(2) (2020) 1--29.

\bibitem{abgrall2021analysis}
R.~Abgrall, J.~Nordstr{\"o}m, P.~{\"O}ffner, S.~Tokareva, Analysis of the
  {SBP-SAT} stabilization for finite element methods part ii: entropy
  stability, Communications on Applied Mathematics and Computation (2021)
  1--23.

\bibitem{svard2014review}
M.~Sv{\"a}rd, J.~Nordstr{\"o}m, Review of summation-by-parts schemes for
  initial--boundary-value problems, Journal of Computational Physics 268 (2014)
  17--38.

\bibitem{fernandez2014review}
D.~C. D.~R. Fern{\'a}ndez, J.~E. Hicken, D.~W. Zingg, Review of
  summation-by-parts operators with simultaneous approximation terms for the
  numerical solution of partial differential equations, Computers \& Fluids 95
  (2014) 171--196.

\bibitem{Abarbanel19811}
S.~Abarbanel, D.~Gottlieb, Optimal time splitting for two- and
  three-dimensional {N}avier-{S}tokes equations with mixed derivatives, Journal
  of Computational Physics 41~(1) (1981) 1--33.

\bibitem{nordstrom2006conservative}
J.~Nordstr{\"o}m, Conservative finite difference formulations, variable
  coefficients, energy estimates and artificial dissipation, Journal of
  Scientific Computing 29~(3) (2006) 375--404.

\bibitem{landau2006}
L.~Landau, E.~Lifshitz, Fluid Mechanics, 2nd ed., Pergamon Press, 2006.

\bibitem{whitham1974}
G.~B. Whitham, Linear and Nonlinear Waves, JWS, 1974.

\bibitem{shallowwaterbook}
C.~B. Vreugdenhil, Numerical Methods for Shallow-Water Flow, Vol.~13, Springer,
  2013.

\bibitem{tomas}
T.~Lundquist, {Private Communication} (2021).

\end{thebibliography}

\end{document}